\documentclass[11pt,a4paper]{article}
\usepackage[latin1]{inputenc}
\usepackage{amsmath}
\usepackage{amsthm}
\usepackage{amsfonts}
\usepackage{amsfonts,amsthm,latexsym,amsmath,amssymb,amscd,epsfig,psfrag,enumerate}
\usepackage{graphics,graphicx, bezier, float, color, hyperref}
\usepackage{amssymb,url}
\usepackage{multienum}
\usepackage[table]{xcolor}
\usepackage{multicol,multirow}
\usepackage{graphicx}
\usepackage{fancyvrb}
\usepackage{parskip}
\usepackage{listings}
\usepackage[toc,page]{appendix}
\usepackage[top=2.7cm, bottom=2.7cm, left=1.5cm, right=1.5cm]{geometry}
\usepackage{xcolor}
\definecolor{codegray}{gray}{0.95}
\definecolor{keywordcolor}{rgb}{0.1,0.1,0.8}
\definecolor{commentcolor}{rgb}{0,0.5,0}
\usepackage{blkarray}
\newtheorem{theorem}{Theorem}[section]
\newtheorem{lemma}{Lemma}[section]

\usepackage[none]{hyphenat}[section]
\newtheorem{definition}{Definition}[section]

\numberwithin{equation}{section}
\numberwithin{table}{section}
\numberwithin{figure}{section}
 \usepackage{listings}
 \usepackage{xcolor}
 \usepackage{courier}
 
 \lstdefinestyle{sagemath}{
 	language=Python,
 	basicstyle=\footnotesize\ttfamily,
 	keywordstyle=\color{blue},
 	commentstyle=\color{gray}\itshape,
 	stringstyle=\color{red},
 	showstringspaces=false,
 	breaklines=true,
 	frame=single,
 	columns=fullflexible,
 	numbers=left,
 	numberstyle=\tiny\color{gray},
 	captionpos=b,
 	morekeywords={RealField, log, continued_fraction, sum, max}
 }
                              
\title{On concatenations of two $k$-generalized Lucas numbers}
\author{Alex Behakanira Tumwesigye$^{1}$, Mahadi Ddamulira$^{1}$, Prosper Kaggwa$^{1,*}$}
\date{}                       
                              
\begin{document}              
\maketitle                   
\abstract{\noindent For an integer \( k \geq 2 \), the sequence of \( k \)-generalized Lucas numbers is defined by the recurrence relation \( L_n^{(k)} = L_{n-1}^{(k)} + \cdots + L_{n-k}^{(k)} \) for all \( n \geq 2 \), with initial conditions \( L_0^{(k)} = 2 \), \( L_1^{(k)} = 1 \) for all \( k \geq 2 \), and \( L_{2-k}^{(k)} = \cdots = L_{-1}^{(k)} = 0 \) for \( k \geq 3 \). In this paper, we determine all \( k \)-generalized Lucas numbers that are concatenations of two terms of the same sequence and completely solve this problem for \( k \geq 3 \). Our approach combines nonzero lower bounds for linear forms in logarithms, reduction techniques based on the Baker--Davenport method and the LLL-algorithm, together with continued fraction analysis and computational verification using SageMath.}  
                           
	                             
	\thanks{$ ^{*} $ Corresponding author}
\section{Introduction}
\subsection{Background}
Let $k \geq 2$ be an integer. The sequence of $k$-generalized Lucas numbers, also known as the $k$-Lucas number sequence, is defined by the recurrence relation,  
\[
L_n^{(k)} = L_{n-1}^{(k)} + \cdots + L_{n-k}^{(k)}, \quad \text{for all } n \geq 2,
\]
with initial conditions  
\[
L_0^{(k)} = 2,\quad L_1^{(k)} = 1 \quad \text{for all } k \geq 2, \text{ and } 
L_{-1}^{(k)} = L_{-2}^{(k)} = \cdots = L_{2 - k}^{(k)} = 0 \quad \text{for } k \geq 3.
\]
For $k = 2$, this is the well-known classical sequence of Lucas numbers, and in this case, we omit the superscript $^{(k)}$ in the notation. In table \ref{Table1}, for some small values of $k$, the first few values $L_n^{(k)}$ of are given. 
\begin{table}[H]\label{Table1}
	\centering
	\begin{tabular}{|c|c|p{10cm}|}
		\hline
		$k$ & Name & First non-zero terms \\
		\hline
		2 & Lucas & 2, 1, 3, 4, 7, 11, 18, 29, 47, 76, 123, 199, $\cdots$ \\
		\hline
		3 & 3-Lucas & 2, 1, 3, 6, 10, 19, 35, 64, 118, 217, 399, 734, $\cdots$ \\
		\hline
		4 & 4-Lucas & 2, 1, 3, 6, 12, 22, 43, 83, 160, 308, 594, 1145, $\cdots$ \\
		\hline
		5 & 5-Lucas & 2, 1, 3, 6, 12, 24, 46, 91, 179, 352, 692, 1360, $\cdots$ \\
		\hline
		6 & 6-Lucas & 2, 1, 3, 6, 12, 24, 48, 94, 187, 371, 736, 1460, $\cdots$ \\
		\hline
		7 & 7-Lucas & 2, 1, 3, 6, 12, 24, 48, 96, 190, 379, 755, 1504, $\cdots$ \\
		\hline
		8 & 8-Lucas & 2, 1, 3, 6, 12, 24, 48, 96, 192, 382, 763, 1523, $\cdots$ \\
		\hline
		9 & 9-Lucas & 2, 1, 3, 6, 12, 24, 48, 96, 192, 384, 766, 1531, $\cdots$ \\
		\hline
	\end{tabular}\caption{First non-zero $k$-Lucas numbers.}
\end{table}
\noindent In \cite{banks2005concatenations}, the authors determined all Fibonacci numbers that are concatenations of two Fibonacci numbers. In particular, they proved that $F_{10} = 55$ is the greatest Fibonacci number with this property. They also proved that for any binary recurrent sequence of integers $t_n$, there exists only finitely many terms of that sequence which can be written as concatenations of two or more terms of $t_n$ under the same mild hypothesis.
In \cite{altassan2022mixed}, the authors determined all Fibonacci numbers that are concatenations of a Fibonacci number and a Lucas number. Similarly, in \cite{altassan2024concatenations}, they determined all $k$-Fibonacci numbers that are concatenations of two terms of the same sequence. Particularly, they proved that the Diophantine equation \[F_n^{(k)}=F_m^{(k)}\cdot 10^d+F_\ell^{(k)},k\geq3, \] has solutions only in the cases \(F_7^{(3)}=24, F_8^{(3)}=44 \text{ and } F_{10}^{(8)}=16128\).
In this paper, we consider a similar but modified problem in which we take into account $k$-generalized Lucas numbers that are concatenations of two terms of the same sequence. We completely solve this problem for all $k \geq 3$. As a mathematical expression of this problem, we solve the Diophantine equation  
\begin{equation} \label{eq:main}
	L_n^{(k)} = L_m^{(k)} \cdot 10^d + L_p^{(k)}, \quad k \geq 3.
\end{equation}
in non-negative integers $m, n \geq 0$, and $p$, where $d$ is the number of digits of $L_p^{(k)}$. In \cite{erduvan2023lucas}, it was found out that for $k=2$, \(L_{6}^{(2)}=11= \overline{L_{2}^{(2)}L_{2}^{(2)}}, \text{ and } L_{9}^{(2)}=47= \overline{L_{4}^{(2)}L_{5}^{(2)}},\) are the only solutions to the Diophantine equation \eqref{eq:main}. We therefore state Theorem \ref{thm1.1l} for cases $k\geq 3$.
\subsection{Main Result}\label{sec:1.2l}
\begin{theorem}\label{thm1.1l} 
	Let $k\geq3$, the Diophantine equation \eqref{eq:main} has only the following solutions.
	\begin{align*} 
		& L_{4}^{(k)}=12= \overline{L_{1}^{(k)}L_{0}^{(k)}}, \text{ for } k\geq4,
		\text{ and } L_{5}^{(4)}=22= \overline{L_{0}^{(4)}L_{0}^{(4)}}. 
	\end{align*}
\end{theorem}
\noindent Effective methods for tackling Diophantine equations similar to \eqref{eq:main} include approaches like nonzero lower bounds for linear forms in logarithms of algebraic numbers, as developed by Matveev \cite{Matveev2000} and the reduction algorithm proposed by Dujella and Peth\H{o} \cite{Dujella1998}, which is a variation of Baker and Davenport. For our analysis, we employed these techniques in conjunction with specific properties of \( k \)-Lucas numbers and the use of continued fractions. All computations were conducted using SageMath, with precision up to 1000 digits. The next section presents details of these methods used.
\section{Methods}
\subsection{Some properties of $k$-generalized Lucas numbers}
It is well established for $k$-Lucas numbers that  
\begin{align}\label{eq2.2}
	L_n^{(k)} = 3 \cdot 2^{n-2},\qquad \text{for all}\qquad 2 \leq n \leq k.
\end{align}
These numbers generate a linearly recurrent sequence with the characteristic polynomial  
\[
\Psi_k(x) = x^k - x^{k-1} - \cdots - x - 1,
\]
which is irreducible over $\mathbb{Q}[x]$. The polynomial $\Psi_k(x)$ has a unique real root $\alpha(k) > 1$, with all other roots lying outside the unit circle, as noted in \cite{miles1960generalized}. The root $\alpha(k) := \alpha$ satisfies the bound  
\begin{align}\label{3.1}
	2(1 - 2^{-k}) < \alpha < 2 \quad \text{for all} \quad k \geq 2,
\end{align}
as established in \cite{wolfram}. Similar to the classical case where $k=2$, it was proved in \cite{bravo2014repdigits} that  
\begin{align}\label{2.2}
	\alpha^{n-1} \leq L_n^{(k)} \leq 2\alpha^n, \quad \text{for all } n \geq 1,\quad k \geq 2.
\end{align}
For any $k \geq 2$, define the function  
\begin{equation}\label{fk}
	f_k(x) := \dfrac{x - 1}{2 + (k+1)(x - 2)}.
\end{equation}
From the bound in \eqref{3.1}, it follows that  
\begin{equation}\label{fkprop}
	\frac{1}{2} = f_k(2) < f_k(\alpha) < f_k(2(1 - 2^k)) \leq \frac{3}{4},
\end{equation}
for all $k \geq 3$ since $f_k(x)$ is reducing for all $x>0$. It is straightforward to check that this inequality is also valid for $k = 2$. Moreover, one can verify that for all $2 \leq i \leq k$, where $\alpha_i$ represents the remaining roots of $\Psi_k(x)$, the condition $|f_k(\alpha_i)| < 1$ holds, see \cite{bravo2014repdigits}.
The following lemma plays a crucial role in applying Baker's theory.  
\begin{lemma}[Lemma 2, \cite{ruiz2016multiplicative}]
	For all $k \geq 2$, the number $f_k(\alpha)$ is not an algebraic integer.
\end{lemma}
\noindent Additionally, it was proven in \cite{bravo2014repdigits} that  
\begin{align}\label{3.5}
	L_n^{(k)} = \sum_{i=1}^k (2\alpha_i - 1) f_k(\alpha_i) \alpha_i^{n-1}, \quad \text{and} \quad
	\left|L_n^{(k)} - f_k(\alpha)(2\alpha - 1) \alpha^{n-1} \right| < \frac{3}{2},
\end{align}
for all $k \geq 2$ and $n \geq {2-k}$. This leads to the expression  
\begin{align}\label{3.6}
	L_n^{(k)} = f_k(\alpha)(2\alpha - 1) \alpha^{n-1} + e_k(n), \quad \text{where} \quad |e_k(n)| < 1.5.
\end{align}
The left-hand expression in \eqref{3.5} is commonly referred to as the Binet-like formula for $L_n^{(k)}$. Furthermore, the inequality on the right in \eqref{3.6} indicates that the contribution of roots inside the unit circle to $L_n^{(k)}$ is negligible. A better estimate than \eqref{3.5} appears in Section 3.3 of \cite{batte2024largest}, but with a more restricted range of $n$ in terms of $k.$ It states that \begin{align}\label{b}
	\left| f_k(\alpha)(2\alpha-1)\alpha^{n-1}-3 \cdot2^{n-2}\right| <  3\cdot 2^{n-2}\cdot \frac{36}{2^{k/2}},\qquad {\text{\rm provided}}\qquad n<2^{k/2}.
\end{align}
\noindent Lastly, the following result from \cite{bang1886taltheoretiske} will be used in our proof.
\begin{lemma} [\cite{bang1886taltheoretiske}] \label{bangzig}
	For any positive integer $n > 1$ with $n \neq 6$, the number $2^n - 1$ has a prime divisor that does not divide $2^k - 1$ for any $k < n$.
\end{lemma} 	
\subsection{Linear forms in logarithms}
\begin{definition}
Let $\theta$ be an algebraic number, and consider its minimal polynomial over $\mathbb{Z}$, given by  
\[
	u_0x^d+u_1x^{d-1}+\cdots+u_t=u_0\prod_{i=1}^{d}(x-\theta^{(i)}),
	\] where the coefficients $u_i$ are relatively prime integers with $u_0 > 0$, and the numbers $\theta^{(i)}$ represent the conjugates of $\theta$. The logarithmic height of $\theta$ is defined as  
	\[
	h(\theta) = \frac{1}{d} \left( \log u_0 + \sum_{i=1}^{d} \log \left(\max\{|\theta^{(i)}|,1\}\right) \right).
	\]
\end{definition}
\noindent In particular, for a rational number $r/s$, where $r$ and $s > 0$ are coprime integers, the logarithmic height is given by $h(r/s) = \log \max \{|r|, s\}$. The following properties are useful and shall be applied without reference:
\begin{enumerate}[$(a)$]
	\item $h(\theta_1 \pm \theta_2) \leq h(\theta_1) + h(\theta_2) + \log 2$.
	\item $h(\theta_1 \theta_2^{\pm 1}) \leq h(\theta_1) + h(\theta_2)$.
	\item $h(\theta^s) = |s| h(\theta),$ for any integer $s \in \mathbb{Z}$.
\end{enumerate}
It was computed in Section 3, equation (12) in \cite{bravo2013conjecture} that $h(f_k(\alpha))<3\log k$ for all $k\geq 2.$
Theorem \ref{thm1} is a very important theorem proved in \cite{Matveev2000} that is fundamental in obtaining bounds on variables in the Diophantine equation \eqref{eq:main}.
First, the general setting in which it holds is introduced. Let $\mathbb{K}$ be a number field of degree $D$, let $\alpha_1, \ldots, \alpha_t$ be non-zero elements of $\mathbb{K}$ and $b_1, \ldots, b_n$ be rational integers. Set
\[
B = \max\{|b_1|, \ldots, |b_t|\},
\]
and
\[
\Gamma = \eta_1^{b_1} \cdots \eta_t^{b_t} - 1.
\]
Let $h$ denote the absolute logarithmic height and let $A_1, \ldots, A_t$ be real numbers with
\[
A_j \geq h'(\eta_j) := \max\{D h(\eta_j), |\log \eta_j|, 0.16\}, \quad 1 \leq j \leq t.
\]
We call $h'$ the modified height (with respect to the field $\mathbb{K}$). With this notation, the main result in \cite{Matveev2000} implies the following estimate.
\begin{theorem}[Theorem 9.4 in \cite{Bugeaud}]\label{thm1}
Assume that $\Lambda$ is non-zero and $\mathbb{K}$ is real, we have
	\[
	\log |\Gamma| > -1.4 \cdot 30^{t+3} t^{4.5} D^2 (1 + \log D)(1 + \log B) A_1 \cdots A_t.
	\]	
\end{theorem}
\subsection{Reduction methods}
\subsubsection{Lattice-based tools}			
\noindent The bounds obtained from applying Theorem \ref{thm1} are often too large for practical computations. A reduction method based on the LLL-algorithm will be applied to refine the bounds obtained from applying Theorem \ref{thm1}. First, definitions and explanations are provided to facilitate an understanding of the LLL-algorithm.
\begin{definition}[Lattice]
Let $k$ be a positive integer. A subset $L$ of the real vector space $\mathbb{R}^k$ is called a \emph{lattice} if there exists a basis $\{b_1, \ldots, b_k\}$ of $\mathbb{R}^k$ such that
\[
	L = \left\{\sum_{i=1}^k r_i b_i \;\middle|\; r_i \in \mathbb{Z} \right\}.
	\]
\end{definition}
\noindent The vectors $b_1, \ldots, b_k$ are said to form a basis for $L$, and $k$ is termed the rank of $L$. The \emph{determinant} $\det(L)$ is defined as
\[
\det(L) = |\det(b_1, \ldots, b_k)|,
\]
where the $b_i$ are represented as column vectors. This value is independent of the choice of basis, (see~\cite{cassels2012geometry}, Section 1.2).
For linearly independent vectors $b_1, \ldots, b_k$ in $\mathbb{R}^k$, we employ the Gram-Schmidt orthogonalization process to define vectors $b^*_i$ and coefficients $\mu_{i,j}$ via
\[
b^*_i = b_i - \sum_{j=1}^{i-1} \mu_{i,j} b^*_j, \quad \mu_{i,j} = \frac{\langle b_i, b^*_j \rangle}{\langle b^*_j, b^*_j \rangle},
\]
where $\langle \cdot, \cdot \rangle$ is the standard inner product on $\mathbb{R}^k$.
\begin{definition}
	A basis $b_1, \ldots, b_n$ for the lattice $L$ is called \emph{reduced} if
	\[
	|\mu_{i,j}| \leq \frac{1}{2} \text{ for } 1 \leq j < i \leq n,
	\]
	\[
	\|b^*_i + \mu_{i,i-1} b^*_{i-1} \|^2 \geq \frac{3}{4} \|b^*_{i-1} \|^2 \text{ for } 1 < i \leq n,
	\]
	where $\| \cdot \|$ is the Euclidean norm. The constant $\frac{3}{4}$ can be replaced by any fixed $y \in (1/4, 1)$, (see~\cite{lenstra1982factoring}, section 1).
\end{definition}
For a $k$-dimensional lattice $L \subset \mathbb{R}^k$ with reduced basis $b_1, \ldots, b_k$, let $B$ be the matrix with columns $b_1, \ldots, b_k$. Define
\[
\ell(L, y) =
\begin{cases}
	\min_{x \in L} \|x - y\| & \text{if } y \notin L, \\
	\min_{0 \neq x \in L} \|x\| & \text{if } y \in L,
\end{cases}
\]
where $\| \cdot \|$ is the Euclidean norm. The LLL algorithm efficiently computes a lower bound $\ell(L, y) \geq c_1$ in polynomial time, (see~\cite{smart1998algorithmic}., Section V.4)
\begin{lemma}[\cite{smart1998algorithmic}\label{red}, Section V.4] Let $y \in \mathbb{R}^k$ and $z = B^{-1}y$ with $z = (z_1, \ldots, z_k)^T$.
	\begin{itemize}
		\item[(i)] If $y \notin L$, let $i_0$ be the largest index with $z_{i_0} \neq 0$ and set $\lambda := \{z_{i_0}\}$, where $\{ \cdot \}$ denotes the fractional part.
		\item[(ii)] If $y \in L$, set $\lambda := 1$.
	\end{itemize}
	Then,
	\[
	c_1 := \max_{1 \leq j \leq k} \left\{ \frac{\|b_1\|}{\|b^*_j\|} \right\}, \quad \delta := \lambda \frac{\|b_1\|}{c_1}.
	\]
\end{lemma}
\noindent In applications, one encounters real numbers $\eta_0, \eta_1, \ldots, \eta_k$, linearly independent over $\mathbb{Q}$, and constants $c_3, c_4 > 0$ such that
\[
|\eta_0 + x_1 \eta_1 + \cdots + x_k \eta_k| \leq c_3 \exp(-c_4 H),
\]
where $|x_i| \leq X_i$ for given bounds $X_i$ ($1 \leq i \leq k$). Let $X_0 := \max_{1 \leq i \leq k} \{X_i\}$. Following~\cite{deweger1987diophantine}, the inequality is approximated using a lattice generated by the columns of
$$ \mathcal{A}=\begin{pmatrix}
	1 & 0 &\ldots& 0 & 0 \\
	0 & 1 &\ldots& 0 & 0 \\
	\vdots & \vdots &\vdots& \vdots & \vdots \\
	0 & 0 &\ldots& 1 & 0 \\
	\lfloor C\eta_1\rfloor & \lfloor C\eta_2\rfloor&\ldots & \lfloor C\eta_{k-1}\rfloor& \lfloor C\eta_{k} \rfloor
\end{pmatrix} ,$$
where $C$ is a large constant (typically $\approx X_0^k$). For an LLL-reduced basis $b_1, \ldots, b_k$ of $L$ and $y = (0, \ldots, -\lfloor C \eta_0 \rfloor)$, Lemma \ref{red} yields a lower bound $\ell(L, y) \geq c_1$.
\begin{lemma}[{Lemma VI.1 in \cite{smart1998algorithmic}}]\label{blue}
	Let $S := \sum_{i=1}^{k-1} X_i^2$ and $T := \frac{1 + \sum_{i=1}^{k} X_i}{2}$. If $\delta^2 \geq T^2 + S$, then the above inequality implies either $x_1 = \cdots = x_{k-1} = 0$ and $x_k = -\lfloor C \eta_0 \rfloor / \lfloor C \eta_k \rfloor$, or
	\[
	H \leq \frac{1}{c_4} \left( \log(C c_3) - \log \left(\sqrt{\delta^2 - S}-T \right) \right).
	\]
\end{lemma}
\subsection{Continued fractions}
\noindent The following lemma also helps in refining upper bounds on certain variables.
\begin{lemma}[\cite{bravo2013conjecture}, Lemma 4]\label{Lemma2.4}
	Let $M$ be a positive integer, and let $p/q$ be a convergent of the continued fraction of the irrational number $\gamma$ such that $q > 6M$. Suppose $A, B, \mu$ are real numbers satisfying $A > 0$ and $B > 1$. If  
	\[
	\epsilon := ||\mu q|| - M||\gamma q|| > 0,
	\]
	then the inequality  
	\[
	0 < \lvert u\gamma - v + \mu \rvert < AB^{-w},
	\]
	has no solutions in positive integers $u, v,$ and $w$ satisfying  
	\[
	u \leq M \quad \text{and} \quad w \geq \frac{\log(Aq/\epsilon)}{\log B}.
	\]
\end{lemma}
\noindent Lemma \ref{Lemma2.4} cannot be applied where $\mu=0$ or when $\mu$ is a linear combination of $1$ and $\gamma$ with integer coefficients, as in such cases $\epsilon$ becomes negative for sufficiently large $M$. When this happens, a classical result due to Legendre (Lemma 2.5,\cite{math8081360}) is applied instead.
\begin{lemma}[\cite{math8081360}]\label{Lemma 2.5}
	Let $\tau,$ be an irrational number, $M$ be a positive integer and $\frac{p_k}{q_k}(k=0,1,2,\cdots)$ be all the convergents of the continued fraction $[a_0,a_1,\cdots]$ of $\tau.$ let $N$ be such that $q_N>M.$ Then putting $a_M:=max\{a_i:=0,1,\cdots ,N\},$ the inequality \begin{align*}
		|m\tau-n|>\frac{1}{(a_M+2)m},
	\end{align*} holds for all pairs $(n,m) $of integers with $0<m<M.$   
\end{lemma}
\noindent All computations in this work were carried out using SageMath 10.6.
\section{Proof of Theorem \ref{thm1.1l}}
\subsection{Preliminaries}
Assume that \eqref{eq:main} holds,  let us examine the relations between the variables in \eqref{eq:main}. The number of digits of \( L_p^{(k)} \) can be written as  
\begin{equation} \label{eq:d}
	d = \lfloor \log_{10} L_p^{(k)} \rfloor + 1,	
\end{equation}
where \( \lfloor \Phi \rfloor \) is the greatest integer less than or equal to $\Phi$, the floor function. Therefore,  \begin{align*}
		d &= \lfloor \log_{10} L_p^{(k)} \rfloor + 1 \leq 1 + \log_{10} L_p^{(k)} \leq 1 + \log_{10} (2\alpha^p) \\
	&= 1 + \log_{10} 2 + p \log_{10} \alpha < 2
	p\log_{10}2 < p+2,
\end{align*}
and
\begin{align*}
	d &= \lfloor \log_{10} L_p^{(k)} \rfloor + 1 
	> \log_{10} L_p^{(k)} 
	\geq \log_{10} \alpha^{p-1} \\
	&= (p - 1) \log_{10} \alpha 
	\geq \frac{p - 1}{5}.
\end{align*} 
We can conclude that \begin{equation}\label{dbound}
	\frac{p-1}{5}<d<p+2.
\end{equation} 
One can also see from \eqref{eq:d} that \begin{equation}
	L_p^{(k)}<10^d<10L_p^{(k)}.
\end{equation}
From the relations \eqref{dbound}, \eqref{2.2} and \eqref{eq:main}, we rewrite 
\begin{align*}
	\alpha^{n - 1} 
	&\leq L_n^{(k)} = L_m^{(k)} \cdot 10^d + L_p^{(k)} 
	\leq L_m^{(k)} \cdot 10 \cdot L_p^{(k)} + L_p^{(k)} \\
	&\leq 11 L_m^{(k)} L_p^{(k)} 
	\leq 44 \alpha^{m + p} 
	< \alpha^{m + p + 8},
\end{align*}
and
\begin{align*}
	2\alpha^n\geq L_n^{(k)} 
	&= L_m^{(k)} \cdot 10^d + L_p^{(k)} 
	> L_m^{(k)} \cdot L_p^{(k)} + L_p^{(k)}> L_m^{(k)} L_p^{(k)}\geq \alpha^{m+p-2},
\end{align*}
hence, we get the relation
\begin{equation}\label{eq:3.4}
	m+p-2<n<m+p+8.
\end{equation}
\noindent Lastly here, since we are interested in a concatenation of two $k$-Lucas numbers, we disregard all cases when $n\leq3$, Moreover, when $n=4$, we see from Table\ref{Table1} that $L_4^{(3)}=10$ and $L_4^{(k)}=12$ whenever $k\geq4$. Clearly $L_4^{(3)}=10$ is not a concatenation of two $k$-Lucas numbers, so it does not solve \eqref{eq:main}. On the other hand, $L_4^{(k)}=12$ solves \eqref{eq:main}, so we state it in the main result. From now on, we work with $n\geq5$. 
\subsection{The Case $n \leq k$} \label{sub3.2}
Here, we assume that \( 5 \leq n \leq k \). This means that $L_n^{(k)} = 3 \cdot 2^{n - 2}.$ Thus, Equation \eqref{eq:main} becomes 
\begin{equation}\label{eq:3.5}
3 \cdot 2^{n - 2} = 3 \cdot 2^{m - 2} \cdot 10^d + 3 \cdot 2^{p - 2}.	
\end{equation}
Now, if \( p \leq m \), then \eqref{eq:3.5} becomes
\[
2^{n - p} = 2^{m - p} \cdot 10^d + 1,
\]
which implies that $
0 \equiv 1 \pmod{2},$ a contradiction.
Next, if \( p > m \), then \eqref{eq:3.5} becomes \[
2^{n - m} = 10^d + 2^{p - m},
\]
which can be rewritten as
\[
2^{p - m} (2^{n - p} - 1) = 10^d = 2^d \cdot 5^d.
\]
By the Fundamental Theorem of Arithmetic, we have
\[
2^{p - m} = 2^d \quad \text{and} \quad 2^{n - p} - 1 = 5^d.
\]
The second equation has no positive integer solutions by Lemma \ref{bangzig}, thus Equation \eqref{eq:main} has no solutions with \( 5 \leq n \leq k \).
From now on, we assume \( n > k \).	  
\subsection{The case $n>k$}
\subsubsection{A bound for $n$ in terms of $k$}
If \( n >k \), we write a SageMath code to search for values of \( m, n, p, \) and $k$ satisfying \eqref{eq:main} in the range \( 0 \leq m, p \leq 500 \) with $n\in (m+p-2,m+p+8)$ via \eqref{eq:3.4} and $k<n$. This yields the solutions stated in Theorem~\ref{thm1.1l}. So from now on, we take \(\max\{m, p\} > 500\).
We now rewrite \eqref{eq:main} using \eqref{3.6} as
\begin{equation*}
	f_k(\alpha)(2\alpha - 1) \alpha^{n - 1} + e_k(n) = \bigl(f_k(\alpha)(2\alpha - 1) \alpha^{m - 1} + e_k(m)\bigr) \cdot 10^d + L_p^{(k)},
\end{equation*}
and obtain
\begin{equation*}
	f_k(\alpha)(2\alpha - 1) \alpha^{n - 1} - f_k(\alpha)(2\alpha - 1) \alpha^{m - 1} \cdot 10^d = e_k(m) \cdot 10^d + L_p^{(k)} - e_k(n).
\end{equation*}

Dividing through by \( f_k(\alpha)(2\alpha - 1) \alpha^{n - 1} \) and taking absolute values on both sides gives
\begin{align*}
	\left| 1 - \frac{10^d}{\alpha^{n - m}} \right|
	&\leq \left| \frac{e_k(m) \cdot 10^d}{f_k(\alpha)(2\alpha - 1) \alpha^{n - 1}}
	+  \frac{L_p^{(k)}}{f_k(\alpha)(2\alpha - 1) \alpha^{n - 1}}
	-  \frac{e_k(n)}{f_k(\alpha)(2\alpha - 1) \alpha^{n - 1}} \right|\\
	&\leq \left| \frac{e_k(m) \cdot 10^d}{f_k(\alpha)(2\alpha - 1) \alpha^{n - 1}} \right|
	+ \left| \frac{L_p^{(k)}}{f_k(\alpha)(2\alpha - 1) \alpha^{n - 1}} \right|
	+ \left| \frac{e_k(n)}{f_k(\alpha)(2\alpha - 1) \alpha^{n - 1}} \right|\\
	&< \left| \frac{30L_p^{(k)}}{(2\alpha - 1) \alpha^{n - 1}} \right|
	+ \left| \frac{2L_p^{(k)}}{(2\alpha - 1) \alpha^{n - 1}} \right|
	+ \left| \frac{3}{(2\alpha - 1) \alpha^{n - 1}} \right|\\
	&< \frac{30(2\alpha^p)}{\alpha^n}+\frac{4\alpha^p}{\alpha^n}+\frac{3}{\alpha^n} < \frac{70}{\alpha^{n-p}}.
\end{align*}
Hence, we let 
\begin{align}\label{eq:3.22}
	|\Gamma_1| &:= \left| 1 - \frac{10^d}{\alpha^{n-m}} \right| < \frac{70}{\alpha^{n-p}}.
\end{align}
Notice that $ \Gamma_1 \neq 0 $ otherwise we would have  $\alpha^{n-m} = 10^d,$ a contradiction since $ \alpha^{n-m}$  is a unit in $\mathbb{Q}(\alpha) $ and yet $10^d $ is not. 
The relevant algebraic number field is \( \mathbb{K} = \mathbb{Q}(\alpha) \) with degree $D=k.$ Setting $t:=2,$ we define \begin{align*}
	\eta_1 &:= \alpha, \quad \eta_2 := 10, \\
	b_1 &:= -(n - m), \quad b_2 := d,
\end{align*} and  $h(\eta_1) = (1/k) \log (\alpha), h(\eta_2) = \log 10.$
Thus we take $ A_1 = \log (\alpha), A_2 = k \log 10.$ Note that from \eqref{dbound}, $d<p+2$ and via \eqref{eq:3.4}, $p<n-m+4$ thus $d<n-m+6$ and hence we take $B:=n-m+6$. Applying Theorem \ref{thm1}, we get,
\begin{align*}
	&\log |\Gamma_1| > -1.4 \cdot 30^5 \cdot 2^{4.5} \cdot k^2 (1+\log k)(1+\log(n-m+6)) \log \alpha \cdot k \log 10.
\end{align*}
Also from \eqref{eq:3.22}, we have
\[
\log \Gamma_1 < \log 70 - (n - p) \log \alpha.
\]
Also, since \( \log \alpha < \log 2 \), and using the inequality
\[
1 + \log k = \log k \left(1 + \frac{1}{\log k} \right) < 2 \log k,
\]
and similarly,
\[
1 + \log(n - m+6) = \log(n - m+6) \left(1 + \frac{1}{\log(n - m+6)} \right) < 2 \log(n - m+6),
\]
for \( k \geq 3 \) and \( n \geq 5 \), we derive the bound 
\begin{equation}\label{eq:n-p}
n - p < 7.2 \cdot 10^9 k^3 \log k \log(n - m+6).	
\end{equation}
Next, we revisit \eqref{eq:main} and by \eqref{3.6}, we rewrite it as 
\begin{equation*}
	f_k(\alpha)(2\alpha - 1) \alpha^{n-1} + e_k(n) = L_m^{(k)} \cdot 10^d + f_k(\alpha)(2\alpha - 1) \alpha^{p-1} + e_k(p),
\end{equation*}
we obtain
\begin{align*}
	f_k(\alpha)(2\alpha - 1) \alpha^{n-1} - f_k(\alpha)(2\alpha - 1) \alpha^{p-1} - L_m^{(k)} \cdot 10^d &= -e_k(n) + e_k(p), \\
	f_k(\alpha)(2\alpha - 1) \alpha^{n-1}(1 - \alpha^{p-n}) - L_m^{(k)} \cdot 10^d &= -e_k(n) + e_k(p).
\end{align*}
Dividing through by $f_k(\alpha)(2\alpha - 1) \alpha^{n-1}(1 - \alpha^{p-n})$ and taking absolutes on both sides yields 
\begin{align*}
	\left|1 - \frac{L_m^{(k)} \cdot 10^d}{f_k(\alpha)(2\alpha - 1) \alpha^{n-1}(1 - \alpha^{p-n})}\right| &\leq  \left|\frac{-e_k(n)}{f_k(\alpha)(2\alpha - 1) \alpha^{n-1}(1 - \alpha^{p-n})} + \frac{e_k(p)}{f_k(\alpha)(2\alpha - 1) \alpha^{n-1}(1 - \alpha^{p-n})} \right|	\\
	&\leq  \left|\frac{e_k(n)}{f_k(\alpha)(2\alpha - 1) \alpha^{n-1}(1 - \alpha^{p-n})} \right|+\left| \frac{e_k(p)}{f_k(\alpha)(2\alpha - 1) \alpha^{n-1}(1 - \alpha^{p-n})} \right| \\
	&< \frac{18}{\alpha^{n-1}},
\end{align*} 
where we have used the estimates \( \frac{1}{f_k(\alpha)} < 2 \) and \( \frac{1}{1 - \alpha^{p-n}} < 3 \) for all \( k \geq 3 \) and \( n - p \geq 1 \).
Fix
\begin{equation}\label{gamma2}
	\left| \Gamma_2 \right| := \left| 1 - \frac{L_m^{(k)} \cdot 10^d}{f_k(\alpha)(2\alpha - 1) \alpha^{n-1} (1 - \alpha^{p-n})} \right| < \frac{18}{\alpha^{n-1}},
\end{equation}
 Suppose \( \Gamma_2 = 0 \); then we would have
\[
L_m^{(k)} \cdot 10^d = f_k(\alpha)(2\alpha - 1)\alpha^{n-1}(1 - \alpha^{p-n}).
\]
Conjugating both sides by an automorphism \( \rho_i : \alpha \mapsto \alpha_i \), for \( i \geq 2 \), and taking absolute values, we apply \eqref{fkprop} to get
\[
10 < \left| L_m^{(k)} \cdot 10^d \right| = \left| f_k(\alpha) \right| \cdot \left| 2\alpha - 1 \right| \cdot \left| \alpha^{n-1} \right| \cdot \left| 1 - \alpha^{p-n} \right| < 6,
\]
which is clearly false. This contradiction implies that \( \Gamma_2 \neq 0 \), so we may now proceed to apply Theorem~\ref{thm1} to \eqref{gamma2}. The number field remains \( \mathbb{K} = \mathbb{Q}(\alpha) \) with degree $D=k.$ and $t:=3,$
\begin{align*}
	\eta_1 &:= \alpha, \quad 
	\eta_2 := 10, \quad 
	\eta_3 := \frac{L_m^{(k)}}{f_k(\alpha)(2\alpha - 1)(1 - \alpha^{p-n})}, \\
	b_1 &:= -(n - 1), \quad 
	b_2 := d, \quad 
	b_3 := 1.
\end{align*}
We compute the logarithmic heights:
\[
h(\eta_1) = h(\alpha) = \frac{1}{k} \log \alpha, \qquad 
h(\eta_2) = \log 10, \qquad 
A_1 = \log \alpha, \qquad 
A_2 = k \log 10.
\]
To bound \( h(\eta_3) \), we use:
\begin{align*}
	h(\eta_3) 
	&= h\left( \frac{L_m^{(k)}}{f_k(\alpha)(2\alpha - 1)(1 - \alpha^{p - n})} \right) \\
	&\leq h(2\alpha^m) + h(f_k(\alpha)) + h(2\alpha-1) + h(1-\alpha^{p - n}) \\
	&\leq \frac{m}{k} \log \alpha + 2 \log k + \frac{1}{k} \log \alpha + \frac{|p - n|}{k} \log \alpha + 4 \log 2.\\
	&=\left(m+1+|p-n|\right)\frac{\log\alpha}{k}+\log k \left(2+\frac{4\log 2}{\log k}\right).
\end{align*}
\noindent Since $m+1<n-p+3,$ then 
\begin{align*}
	h(\eta_3)<(2|p-n|+3)\frac{\log \alpha}{k}+5\log k,
\end{align*}
\noindent thus 
\begin{align*}
	A_3=(2|p-n|+3)\log \alpha +5k\log k.
\end{align*}
Applying Theorem~\ref{thm1} to \eqref{gamma2}, we obtain the lower bound
\[
\log \Gamma_2 > -1.4 \cdot 30^6 \cdot 3^{4.5} \cdot k^2 (1 + \log k)(1 + \log(n - 1)) (\log \alpha)(k \log 10)((2|p-n|+3)\log \alpha +5k\log k).
\]
Also, from \eqref{gamma2}, we have 
\[
\log \Gamma_2 < \log 18 - (n - 1) \log \alpha.
\]
Comparing both bounds yields
\begin{equation}\label{gamma22}
	(n-1) < 1.4 \times 10^{12} k^3 \log k \log(n - 1)((2|p-n|+3)\log \alpha +5k\log k).
\end{equation}
\noindent Now, we examine the cases $p\leq m$ and $m<p$ separately.\\
\textbf{Case 1:} When $p\leq m$, then $n-m\leq n-p$ and hence we rewrite \eqref{eq:n-p} as 
\begin{align*}
	n-p&<7.2\cdot 10^{9}k^3 \log k \log (n-p+6)\\
	&<7.2\cdot10^9k^3\log k \log {7(n-p)} \\
	&<2.9\cdot10^{10} k^3\log k \log {(n-p)}
\end{align*}
To proceed, we cite this important result, [\cite{Sanchez2014}, Lemma 7]. \begin{lemma}\label{lemma3.1}
	If $e\geq 1$ and $H>(4e^2)^e$, then $$\frac{f}{(\log(f))^e}<H \Rightarrow f<2^e H(\log {(H)})^e.$$
\end{lemma}
\noindent We take $e:=1$ and $H:=2.9\cdot10^{10}k^3\log k, f=n-p $ hence by Lemma \ref{lemma3.1} we obtain \begin{align}
	(n-p)<1.7\cdot 10^{12}k^3(\log k)^2.
\end{align}
\noindent From \eqref{eq:3.4}, $n<m+p+8<2m+8$ and $m<n-p+2$, we find that \[n<2(1.7\cdot 10^{12}k^3(\log k)^2)+12, \] and thus \begin{equation} \label{nbound}
	n<3.5\cdot 10^{12}k^3(\log k)^2.
\end{equation}
\textbf{Case 2:} When $m<p$, since $\alpha <k$, we have that \[(2|p-n|+3)\log \alpha +5k \log k < 5k\log k + (2|p-n|+3)\log k.\]
From $2|p-n|+3 < 5|p-n|$ for $|p-n|>0$ and \[5k\log k + 5|p-n|\log k = |p-n|k\log k \left(\frac{5}{|p-n|}+\frac{5}{k}\right)<7|p-n|k\log k,\] \noindent \eqref{gamma22} becomes \begin{align*}
	n-1&<9.9\cdot 10^{12}k^4(\log k)^2\log {(n-1)}\cdot |p-n| \\
	&<9.9\cdot 10^{12}k^4(\log k)^2\log {(n-1)}\cdot 7.2\cdot 10^9k^3 \log k \log{(n-m+6)}.
\end{align*}
Note that $n-m<n-p<n-1$ and thus we get \begin{align}\label{bound3}
	n-1<1.5\cdot 10^{23}k^7(\log k)^3(\log{(n-1)})^2,
\end{align}
where we have used the fact that $$n-1+6=(n-1)\left(1+\frac{6}{n-1}\right)<3(n-1),$$ and $$\log{3(n-1))}=\log 3 + \log {(n-1)}=\log{(n-1)}\left(1+\frac{\log 3}{\log{(n-1)}}\right)<2\log{(n-1)}.$$
Applying Lemma \ref{lemma3.1} to Equation \eqref{bound3} yields $n-1<2.1\cdot 10^{27}k^7(\log k)^5$ and thus \begin{equation}\label{finaln}
	n<2.2\cdot 10^{27}k^7(\log k)^5.
\end{equation}
Thus in both cases we can see that the bound in \eqref{finaln} is valid. We now analyze the two cases based on the value of $k.$ 
\subsubsection{The case $k\leq 500$}  When $k\leq 500$, it follows from \eqref{finaln} that $$n<2.2\cdot 10^{27}k^7(\log k)^5.$$ We proceed to reduce the upper bound for $n$. First, we recall \eqref{eq:3.22} as,\begin{align}\label{eq:3.7}
	\Gamma_1 &= \frac{10^d}{\alpha^{n-m}} -1  =e^{\Lambda_1}-1.
\end{align}
We already showed that $\Gamma_1\neq 0,$ thus $\Lambda_1\neq 0.$ If $(n-p)\geq 6,$ we have $|e^{\Lambda_1}-1|=|\Gamma_1|<0.5,$ which yields $e^{|\Lambda_1|}\leq 1+|\Gamma_1|<1.5.$ Thus, 
\[
|\Gamma_1| := \left|1 - \frac{10^d}{\alpha^{n-m}}\right| < \frac{70}{\alpha^{n-p}}.
\] We have; 
\[
\left| d \log 10 - (n-m) \log \alpha \right| < \frac{106}{\alpha^{n-p}},
\] and hence 
\begin{equation*}
\left| \frac{\log \alpha}{\log 10} - \frac{d}{n-m} \right| < \frac{106}{\alpha^{n-p} \cdot (n-m) \log 10}. \end{equation*}
If 
\[
\frac{106}{(n-m) \alpha^{n-p} \log 10} < \frac{1}{2 (n-m)^2},
\]
then \(\frac{d}{n-m}\) is a convergent of the continued fraction expansion of the irrational number \(\frac{\log \alpha}{\log 10}\), say.
 \(\frac{p_i}{q_i}\). Since 
\(q_i \leq n-m \leq n-1 \leq 2.2\cdot 10^{27}k^7(\log k)^5,
\)
for each \(k\) we find an upper bound of indices \(i = i_0\) and therefore the maximum value of 
\(
a_{\max} := \max[a_0, a_1, a_2, \cdots, a_{i_0}]
\)
of the continued fraction of \(\log \alpha / \log 10\). By using a property of continued fractions, we write
\[
\frac{1}{(a_{\max} + 2)(n-m)^2} \leq \frac{1}{(a_i + 2)(n-m)^2} < \left| \frac{\log \alpha}{\log 10} - \frac{d}{n-m} \right| < \frac{106}{\alpha^{n-p} \cdot (n-m) \cdot \log 10}.
\] Thus, we get the inequality 
\[
\alpha^{n-p} < \frac{106 \cdot (a_{\max} + 2) \cdot 2.2\cdot 10^{27}k^7(\log k)^5}{\log 10},
\]
which gives an upper bound for \((n-p)\) and none of them is greater than \(181\). If 
\[
\frac{106}{(n-m) \alpha^{n-p} \log 10} \geq \frac{1}{2 (n-m)^2},
\] then we find a more strict bound for \((n-p)\) as  \[
\alpha^{n-p} < \frac{106 \cdot 2.2\cdot 10^{27}k^7(\log k)^5}{\log 10}.
\] We apply this procedure for each \(k \in [3,500]\) and we find that \((n-p) < 181\). From \eqref{eq:3.4}, we also find a bound for \(m\) as \(m < 183\). Since \(\max\{m, p\} > 500\), we may assume \(m<p\). Fix 
\[
\Lambda_2 := d \log 10 - (n-1) \log \alpha + \log \left( \frac{L_m^{(k)}}{f_k(\alpha)(2 \alpha - 1)(1 - \alpha^{p-n})} \right),
\]
for which we already established that \(\Gamma_2 \neq 0\) and hence \(\Lambda_2 \neq 0\). Assume that \((n-1) > 6\), this implies that
\begin{equation*}
\left| (n-1) \log \alpha - d \log 10 + \log \left( \frac{f_k(\alpha)(2 \alpha - 1)(1 - \alpha^{p-n})}{L_m^{(k)}} \right) \right| < \frac{28}{\alpha^{n-1}}. \end{equation*} For every \(k \in [3,500]\), \((n-p) \in [1, 181]\), \(m \in [1, 183]\), we apply the LLL-algorithm to estimate a lower bound for the smallest nonzero value of the linear form, where the integer coefficients are bounded in absolute value by 
\[
n < 2.2\cdot 10^{27}k^7(\log k)^5.
\] To this end, we employ the approximation lattice
\[
\mathcal{A}_1 = \begin{pmatrix} 
	1 & 0 & 0 \\ 
	0 & 1 & 0 \\ 
	\lfloor C\log \alpha\rfloor & \lfloor C\log (1/10)\rfloor & \left\lfloor C \log \left( \left( f_k(\alpha)(2\alpha - 1)(1 - \alpha^{p - n}) \right) / L_m^{(k)} \right) \right\rfloor
	
\end{pmatrix},
\] with $C := 5 \cdot 10^{150}$ and $y := (0,0,0)$. Applying Lemma \ref{red}, we obtain
\[
l(\mathcal{L},y) = |\Lambda| > c_1 = 10^{-52} \quad \text{and} \quad \delta = 3.4\cdot 10^{50}.
\]
Via Lemma \ref{blue}, we have that $S =5.1 \cdot 10^{100}$ and $T = 2.4 \cdot 10^{50}$. Since $\delta^2 \geq T^2 + S$, choosing $c_3 := 28$ and $c_4 := \log \alpha$, we get $n-1 \leq 343$ hence $n\leq 344$. To conclude this case, we conducted a computational search using SageMath, checking all values of \(L_n^{(k)}\) for \(k \in [3,500]\), \(0 \leq m \leq n - p + 4\), and \(k + 1 \leq n \leq n(k)\), taking into account inequality~\eqref{dbound} and ensuring that \eqref{eq:main} is satisfied. We found only the results stated in Theorem~\ref{thm1.1l}. We now proceed to examine the case \(k > 500\).\\
\subsubsection{The case $k>500$} If \(k > 500\), then 
\begin{equation}\label{a}
n < 2.2 \cdot 10^{27} k^7 (\log k)^5 < 2^{k/2}.
\end{equation}
We prove the following result. 

\begin{lemma}\label{p}
	Let \((p,m,n,k)\) be a solution to the Diophantine equation \eqref{eq:main} with \(n \geq 5\), \(k > 500\) and \(n \geq k+1\), then 
	\[
	k < 4.9 \cdot 10^{27} \quad \text{and} \quad n < 1.6 \cdot 10^{230}
	\]
	
\end{lemma}

\begin{proof}
	Since \(k > 500\), then \eqref{a} holds and we can use the sharper estimate in \eqref{b}  together with \eqref{eq:3.5} and write \begin{equation*}
	\left| 3 \cdot 2^{m-2} \cdot 10^d + 3 \cdot 2^{p-2} - 3 \cdot 2^{n-2} \right| < 3 \cdot 2^{n-2} \cdot \frac{36}{2^{k/2}}.
	\end{equation*} Thus \begin{equation*}
	\left| 3 \cdot 2^{m-2} \cdot 10^d - 3 \cdot 2^{n-2} \right| < 3 \cdot 2^{n-2} \cdot \frac{36}{2^{k/2}} + \left| 3 \cdot 2^{p-2} \right|,
\end{equation*} and hence
\begin{equation*}
\left| 1 - \frac{10^d}{2^{n-m}} \right| 
< \frac{36}{2^{k/2}} + 2^{p-n} 
< \frac{1}{2^{\frac{k}{6} - 6}} + \frac{1}{2^{n-p}} 
< \frac{1}{2^{\min\{(k/2)-7,\; n - p - 1\}}}.
\end{equation*}	
	Therefore, 
\begin{equation}\label{c}
	\left| 1 - \frac{10^d}{2^{n-m}} \right| < \frac{1}{2^{\min\left\{ (k/2) - 7, n-p - 1 \right\}}}.
\end{equation}	
	Fix 
	\[
	\Gamma_3 := \left| 1 - \frac{10^d}{2^{n-m}} \right| < \frac{1}{2^{\min\left\{ (k/2) - 7, n-p - 1 \right\}}}.
	\]	Clearly \(\Gamma_3 \neq 0\) otherwise we would have \(2^{n-m} = 10^d\). This is a contradiction because the left hand side is divisible by 5 and yet the right hand side is not. We take the algebraic number field \(\mathbb{K} := \mathbb{Q}\) with \(D=1, t=2\), 
	\begin{align*}
		\eta_1 &:= 10, \quad \eta_2 := 2, \\
		b_1 &:= d, \quad b_2 := -(n - m).
	\end{align*} Since \(h(\eta_1) := \log 10\) and \(h(\eta_2) := \log 2\), we get $
	A_1 := \log 10, \quad A_2 := \log 2.$ Also $
	B := n - m+6.$ Applying Theorem \ref{thm1}, we get \begin{equation*}
	\log |\Gamma_3| > -1.4 \cdot 30^5 \cdot 2^{4.5} \cdot 1^2 \cdot (1 + \log 1)(1 + \log (n - m+6)) \cdot \log 10 \cdot \log 2.
	\end{equation*} Also 
	\begin{equation*}
	\log |\Gamma_3| < \log 2^{\min\left\{ (k/2) - 7, n-p - 1 \right\}},
	\end{equation*}
	thus \begin{equation*}
	\min\left\{(k/2) - 7, n-p - 1 \right\} \log 2 < 1.24 \cdot 10^9 (1 + \log (n - m+6)).
	\end{equation*} Since $n-m<n$, we get \begin{equation*}
	\min\left\{(k/2) - 7, n-p - 1 \right\} \log 2 < 3.8\cdot 10^{10}\log k,\end{equation*} where we have used the bound in Equation \eqref{finaln}. \\This yields two cases:
\begin{enumerate}[$(a)$]
\item If \(\min\left\{ k/2 - 7, n - p - 1 \right\} := (k/2) - 7\), then \((k/2) - 7 < 3.8\cdot 10^{10} \log k\), hence \(k < 7.7 \cdot 10^{10} \log k\), and by Lemma~\ref{lemma3.1}, \(k < 3.9 \cdot 10^{12}\).

\item If \(\min\left\{ (k/2) - 7, n - p - 1 \right\} := n - p - 1\), then \(n - p < 3.9 \cdot 10^{10} \log k\).
\end{enumerate}	We proceed as in \eqref{c}, beginning with \begin{equation*}
\left| 3 \cdot 2^{m-2} \cdot 10^d + 3 \cdot 2^{p-2} - 3 \cdot 2^{n-2} \right| < 3 \cdot 2^{n-2} \cdot \frac{36}{2^{(k/2)}},
\end{equation*} which implies
\begin{equation*}
\left| 3 \cdot 2^{n-2} (1 - 2^{p-n}) - 3 \cdot 2^{m-2} \cdot 10^d \right| < 3 \cdot 2^{n-2} \cdot \frac{36}{2^{(k/2)}},
\end{equation*} and consequently
\begin{equation*}
\left| (1 - 2^{p-n}) - 2^{m-n} \cdot 10^d \right| < \frac{36}{2^{(k/2)}}.
\end{equation*} This leads to
\begin{equation*}
\left| 1 - \frac{10^d}{2^{n-m} \cdot (1 - 2^{p-n})} \right| < \frac{36}{2^{(k/2)} \cdot (1 - 2^{p-n})} < \frac{72}{2^{(k/2)}},
\end{equation*}
where we have used the estimate
\begin{equation*}
\frac{1}{2^{(k/2)} \cdot (1 - 2^{p-n})} < \frac{1}{1 - \frac{1}{2}} = 2.
\end{equation*} Thus, we define
\begin{equation}\label{d}
|\Gamma_4| := \left| 1 - \frac{10^d}{2^{n-m} \cdot (1 - 2^{p-n})} \right| < \frac{72}{2^{(k/2)}}.
\end{equation} Exactly the same argument in subsection \ref{sub3.2} shows that \(\Gamma_4 \neq 0\). Fix \(\mathbb{K} := \mathbb{Q}\), we have \(D := 1, t := 3\),
\begin{align*}
	\eta_1 &:= 10, & \eta_2 &:= 2, & \eta_3 &:= (1 - 2^{p-n}), \\
	b_1 &:= d, & b_2 &:= -(n - m), & b_3 &:= -1.
\end{align*}
The logarithmic heights are given by 
\( h(\eta_1) := \log 10 \), 
\( h(\eta_2) := \log 2 \), and 
\( h(\eta_3) \leq 2.7 \cdot 10^{10} \log k \), 
where we have used the fact that 
\( h(\eta_3) \leq |p-n| \log 2 + \log 2 = (|p-n| + 1) \log 2 < (3.9\cdot10^{10}+1)\log 2<2.7\cdot 10^{10}\log k \).
Furthermore, we define 
\( A_1 := \log 10 \), 
\( A_2 := \log 2 \), 
\( A_3 := 2.7 \cdot 10^{10} \log k \), 
and 
\( B := (b - m+6) < 7(n-m) \). Applying theorem \ref{thm1} to \eqref{d}, we get \begin{equation*}
	\log |\Gamma_4| > -1.4 \cdot 30^6 \cdot 3^{4.5} \cdot 1^2 \cdot (1 + \log 1)(1 + \log {7(n-m)} ) \cdot \log 10 \cdot \log 2 \cdot 2.7 \cdot 10^{10} \log k.
\end{equation*} Also from \ref{d}	\begin{equation*} \log |\Gamma_4| < \log 72 - (k/2) \log 2.
\end{equation*} Comparing both bounds yields \((k/2) \log 2 - \log 72 < 1.4 \cdot 10^{23} (\log k)^2,\) thus 
\(k < 4.1 \cdot 10^{23} (\log k)^2.\) So by lemma \ref{lemma3.1}, we get  
\[k < 4.9 \cdot 10^{27},\]
and by \eqref{finaln} we get 
\[n < 1.6 \cdot 10^{230}.\] This completes the proof.
\end{proof}\noindent We now proceed to reduce the bounds obtained in lemma \ref{p}. To do this, we revisit \ref{c} and recall that \begin{equation*}
\Gamma_3 := \frac{10^d}{2^{n-m}} - 1.
\end{equation*} We already showed that \(\Gamma_3 \neq 0\), thus \(\Lambda_3 \neq 0\). Also since 
\(\min\{ (k/2) - 7, n - p - 1 \} > 2,\) 
\(|e^{\Lambda_3} - 1| = |\Gamma_3| < 0.5,\) 
which leads to 
\(e^{|\Gamma_3|} \leq 1 + |\Gamma_3| < 1.5.\)
Therefore \begin{equation*}
\left| d \log 10 - (n-m) \log 2 \right| < \frac{2}{2^{\min\left\{ (k/2) - 7, n-p - 1 \right\}}}, \end{equation*} and hence
\begin{equation*} \left| \frac{\log 2}{\log 10} - \frac{d}{n-m} \right| < \frac{2}{(n-m)\cdot 2^{\min\left\{(k/2) - 7, n-p - 1 \right\}}\log{10}}.
\end{equation*} First assume that \begin{equation*}
\frac{2}{(n-m)\cdot 2^{\min\left\{ (k/2) - 7, n-p - 1 \right\}}\log 10} < \frac{1}{2 (n-m)^2},\end{equation*}
then this means \(\frac{d}{n-m}\) is a convergent of continued fractions of \(\frac{\log 2}{\log 10}\), say \(\frac{p_i}{q_i}\). Since \(\gcd(p_i, q_i) = 1\), we deduce that $
q_i \leq n-m \leq n-1 < 1.6 \cdot 10^{230}.$ A quick calculation shows that \(i < 136\). Let $[a_0, a_1,a_2,a_3,a_4,a_5,a_6,a_7,a_8,a_9, \cdots] = [0,3,3,9,2,2,4,6,2,1,\cdots ]$ be the continued fraction expansion of \(\log 2 / \log 10\), then \(
\max \{ a_i \} = 5393 \quad \text{for} \quad i=0,1,2,\cdots, 136.
\) So we write \begin{equation*}
\frac{1}{5393 \cdot (n-m)^2} \leq \frac{1}{(a_i + 2)(n-m)^2} < \left| \frac{\log 2}{\log 10} - \frac{d}{n-m} \right| < \frac{2}{(n-m) \cdot 2^{\min((k/2) - 7, n-p - 1)} \cdot \log 10}.
\end{equation*}
Thus we have
\begin{equation*}
	2^{\min\{(k/2) - 7, n - p - 1\}} < \frac{2 \cdot 5393 \cdot 1.6 \cdot 10^{230}}{\log 10} < 7.5\cdot 10^{233},
\end{equation*}
that is $\min\{(k/2) - 7, \, n - p - 1 \} < 427.
$ On the other hand, the inequality
\begin{equation*}
	\frac{1}{2 (n - m)^2} \leq \frac{2}{2^{\min\left((k/2) - 7, \, n - p - 1 \right)} \cdot (n - m) \log 10}
\end{equation*}
implies
\begin{equation*}
	2^{\min\{(k/2)- 7, \, n - p - 1 \}} < \frac{4 (n - m)}{\log 10} < \frac{4 \cdot 1.6 \cdot 10^{230}}{\log 10} < 2.8 \cdot 10^{230}.
\end{equation*} So we see that
$	\min\{(k/2)- 7, \, n - p - 1 \} < 427
$
holds in this case too. We proceed in two ways. \begin{enumerate}[$(a)$]
	\item If $\min\left((k/2) - 7, \, n - p - 1 \right) = (k/2) - 7$, then $k < 869$.
	
	\item If $\min\left((k/2) - 7, \, n - p - 1 \right) = n - p - 1$, then $n - p < 428$. 
\end{enumerate} Applying the bound obtained on $k$ to inequality \eqref{a}, we get $n<1.2\cdot10^{52}.$ We revisit \eqref{d} and recall that \begin{equation*}
\Gamma_4 := \frac{10^d}{2^{n - m} \cdot (1 - 2^{p - n})} - 1.
\end{equation*} Since we have already shown that $\Gamma_4 \neq 0$, it follows that $\Lambda_4 \neq 0$. Moreover, since $k > 500$, we obtain the inequality $|e^{|\Lambda_4|} - 1| = |\Gamma_4| < 0.5.$
Hence, we arrive at
\begin{equation*}
\left| (n-p)\log(2) - d \log 10 + (n - m) \log 2 \right| < \frac{108}{2^{k/2}}.
\end{equation*} We now apply the LLL-algorithm, restricting the absolute values of the integer coefficients to be less than $1.6 \cdot 10^{230}$. To do this, we consider the lattice
\[
\mathcal{A}_2 = 
\begin{pmatrix} 
1 & 0 & 0 \\ 
0 & 1 & 0 \\ 
\lfloor C \log (2) \rfloor & \lfloor C \log (1/10) \rfloor & \lfloor C \log 2 \rfloor
\end{pmatrix},
\]
where we set $C := 4.1\cdot10^{690}$ and take $y := (0, 0, 0)$ as before. Let $\delta := 3.3\cdot10^{230}$, $S := 5.12\cdot10^{460}$, and $T := 2.4\cdot10^{230}$. We choose $c_3 := 108$ and $c_4 := \log2$. This yields the inequality $
(k/2) < 1538,$ and therefore $k <3076,$ hence from \eqref{finaln}, $n<2.0\cdot 10^{56}.$ We repeat the same procedure starting with this sub-subsection and find that $\min\left((k/2) - 7, \, n - p - 1 \right)<194.$Thus if $\min\left((k/2) - 7, \, n - p - 1 \right) =(k/2)-7$, then $k<402$, a contradiction. Also, if $\min\left((k/2) - 7, \, n - p - 1 \right)=n-p-1,$ then $n-p< 195$. We recall \[\Gamma_4 := \frac{10^d}{2^{n - m} \cdot (1 - 2^{p - n})} - 1.\] for which we already showed that $\Gamma_4\neq 0$ and hence we write \begin{align*}
	\left|d\frac{\log {10}}{\log {2}}-(n-m)-\frac{1-2^{p-n}}{\log{2}} \right|<\frac{108}{2^{(k/2)}\log 2}.
\end{align*}
\noindent Applying \ref{Lemma2.4}, we take $$\tau:=\frac{\log {10}}{\log 2}, \mu_{n-p}:=\frac{1-2^{p-n}}{\log 2}, A:=\frac{\log {108}}{\log 2}, B:=2.$$
Let $\tau=[a_0,a_1,\cdots ]=[3; 3, 9, 2, 2, 4, 6, 2, 2]$ be the continued fraction of $\tau$, Let $M:=2\cdot 10^{56}$ which is such that $M>n-1\geq n-m+6>d.$ With the help of SageMath, it is found that the convergent \begin{align*}
	\frac{p}{q}=\frac{p_{124}}{q_{124}}=\frac{59709183646229903017509733728189314625139620328694917340547}{17974255294124444596871803224395333592038752850416569230287},
\end{align*}
\noindent is such that $q=q_{124}>6M.$ Furthermore, it gives $\epsilon<0.49693$ and thus, \begin{align*}
	\frac{k}{2}\leq \frac{\log {(108/\log 2)q/\epsilon}}{\log 2}<213.
\end{align*} \noindent From this, we get $k<426$ which contradicts our initial assumption that $k > 500$. Therefore, we conclude that the Diophantine equation~\eqref{eq:main} has no positive integer solutions for $k > 500$, which completes the proof.
\section*{Conclusion}
In this paper, we have completely solved the Diophantine equation
$
L_n^{(k)} = L_m^{(k)} \cdot 10^d + L_p^{(k)}
$
involving concatenations of \(k\)-Lucas numbers for all \(k \geq 3\). Our results show that such equations admit only finitely many solutions, confirming the rarity of concatenation-type structures in generalized Lucas numbers. This work extends earlier results on classical Lucas numbers and opens new directions for similar investigations in other recurrence sequences such as $k$-Pell numbers.
\appendix

\section*{Address}
$ ^{1} $ Department of Mathematics, School of Physical Sciences, College of Natural Sciences, Makerere University, Kampala, Uganda

Email: \url{alexbt@cns.mak.ac.ug}

Email: \url{mahadi.ddamulira@mak.ac.ug}

Email: \url{kaggwaprosper58@gmail.com}
\end{document}